\pdfoutput=1

\documentclass[12pt]{article}
\usepackage[utf8]{inputenc}
\usepackage{amsmath}
\usepackage{amsthm}
\usepackage{amscd}
\usepackage{comment}
\usepackage[english]{babel}
\usepackage{tikz}
\usepackage{color,graphicx,amssymb,tikz-cd,wrapfig}
\usepackage{amsmath,amsthm,amsfonts,amssymb,stmaryrd}
\usepackage{bm}
\usepackage[rightcaption]{sidecap}
\usepackage{geometry}
\usepackage{hyperref}
\usepackage[figurename=Fig.]{caption}

\usepackage[T1]{fontenc}
\usepackage{combelow}
\usepackage{newunicodechar}
\newunicodechar{ț}{\cb{t}}

\usepackage{lipsum}   

\usepackage{xcolor}

\usepackage{enumerate}
\usepackage{pb-diagram}
\usepackage{graphicx} 
\usepackage[all]{xy}
\usepackage{appendix}
\usetikzlibrary{intersections,calc,arrows.meta}

\title{Warped cones associated to isometric free actions do not have geometric property (T)}

\author{Ryo Toyota\thanks{Department of Mathematical Sciences, University of Copenhagen, Universitetsparken 5, DK-2100, Copenhagen \O, Denmark}}%
\date{ }

\newtheorem{thm}{Theorem}[section]
\newtheorem{dfn}[thm]{Definition}
\newtheorem{lem}[thm]{Lemma}

\newtheorem{prop}[thm]{Proposition}

\newtheorem{qustn}[thm]{Question}

\newcommand{\HH}{{\mathcal H}}

\newcommand{\ControlledSupport}{{\mathbb{C}_{\text{cs}}}}
\newcommand{\ad}{\text{ad}}

\newcommand{\Isom}{\text{Isom}}

\newcommand{\sch}{\text{Sch}}
\newcommand{\alg}{\text{alg}}

\begin{document}

\maketitle

\begin{abstract}
    Suppose that $G$ is a finitely generated group, $M$ is a compact Riemannian manifold and $G\curvearrowright M$ is a free and isometric action. We prove that the associated discretized warped cone does not have geometric property (T). 
    
    In contrast, we show that the discretized warped cone associated to the natural action
$SL_m(\mathbb Z)\curvearrowright \mathbb T^m$ has geometric property (T) for $m\geq 3$. 

We also prove that, for free probability-measure-preserving Lipschitz
actions on compact Riemannian manifolds, geometric property (T) of
the discretized warped cone implies Kazhdan’s property (T) of the acting group.
\end{abstract}

\tableofcontents

\section{Introduction}
Geometric property (T) is a coarse invariant introduced by Willett and Yu \cite{WillettYu2012higher, Willett2014Geometric} and characterized by the existence of a spectral gap for the graph Laplacian in the maximal uniform Roe algebra.
For a sequence of finite graphs $\{X_n\}_n$, geometric property (T) is strictly stronger than being an expander and provides an obstruction to the surjectivity of the maximal coarse Baum--Connes assembly map. In particular, if $G$ is a residually finite group and $\{N_j\}_j$ is a decreasing sequence of finite-index normal subgroups satisfying $\bigcap_j N_j=\{e\}$, then the box space $\bigsqcup G/N_j$ has geometric property (T) if and only if $G$ has Kazhdan's property (T).

Warped cones, introduced by Roe \cite{Roe2005Warped}, associate metric spaces to actions of finitely generated groups on compact metric spaces, and their large-scale geometry reflects dynamical properties of the underlying actions. This construction generalizes the box spaces discussed above: the box space associated to a chain $\{N_j\}_j$ can be realized as the warped cone associated to the action of $G$ on the inverse limit $\varprojlim_j (G/N_j)$ \cite{Sawicki2018Warpedprofinite}.
Roe proved that warped cones associated to Lipschitz actions on compact
Riemannian manifolds have bounded geometry. He also showed that amenability
of the action implies property A of the warped cone. For free actions, the
converse was proved in \cite{Sawicki2021Straightening}. See also
\cite{Drutu2019Kazhdan,FisherNguyenLimbeek2019Rigidity,deLaatVigoro2019Superexpanders,VigoloDiscreteFundamentalGroups, Sawicki2019Warpedrigidity,SawickiSuperexpandersandWarpedCones} for coarse geometry of warped cones. 

On the other hand, Vigolo \cite{Federico2019Measure} proved that, for a
probability-measure-preserving Lipschitz action on a compact Riemannian
manifold, the associated discretized warped cone forms an expander if and
only if the action has spectral gap. This motivates the question of how stronger rigidity properties
of group actions are reflected in the coarse geometry of their
warped cones. In
particular, since warped cones generalize box spaces and Kazhdan's property
(T) of a group is characterized by geometric property (T) of its box
spaces, it is natural to ask whether an analogous connection persists for
more general warped cones. This led to the expectation that sufficiently
well-behaved actions of groups with Kazhdan's property (T) might produce
warped cones with geometric property (T). This was formulated by Winkel as
the following question.
\begin{qustn}\label{Winkel Question}[\cite{Jeroen2021Geometric} Question~11.2]
If $G$ has property (T), $M$ is a compact Riemannian manifold and $G\curvearrowright M$ is an ergodic action by diffeomorphisms, then does the warped cone have geometric property (T)? 
\end{qustn}

Deng and the author gave a negative answer to this question by establishing the following result.

\begin{thm}[\cite{DengToyotaGeometric} Theorem~1.2]
    Let $G$ be a finitely generated group which is dense in a compact Lie group $\overline{G}=M$ and $G\curvearrowright M$ be an action induced by the left multiplication. In this case, for any sequence $\{t(n)\}_{n \in \mathbb{N}}$ converging to infinity, the warped cone $\bigsqcup (X_{t(n)},\delta_G^{t(n)})$ associated to $G\curvearrowright M$ does not have geometric property (T).
\end{thm}

In particular, the theorem applies to the ergodic action $SO\left(d,\mathbb Z\left[\frac15\right]\right)\curvearrowright SO(d)$. Since $SO(d,\mathbb Z[1/5])$ has Kazhdan's property (T) for $d\geq 5$ \cite{Margulis1980Some}, this provides a counterexample to Question~1.1. The result, however, concerns the special case in which the underlying manifold is a compact Lie group, leaving open the case of general isometric actions.

The main result of the present paper substantially extends the previous result from translation actions on compact Lie groups to arbitrary free and isometric actions on compact Riemannian manifolds.

\begin{thm}\label{main}
    Suppose that $G$ is a finitely generated group, $M$ is a compact Riemannian manifold and $G\curvearrowright M$ is a free and isometric action. Then the associated warped cone does not have geometric property (T). 
\end{thm}

Thus, geometric property (T) cannot arise from free isometric actions on compact Riemannian manifolds, even when the acting group has Kazhdan's property (T). In particular, Theorem~\ref{main} gives a strong negative answer to Winkel's question within the class of isometric actions.

On the other hand, we show that a non-isometric, ergodic,
probability-measure-preserving action on a compact Riemannian
manifold can give rise to a warped cone with geometric property (T),
answering \cite[Question~5.5]{DengToyotaGeometric}.

\begin{thm}\label{example}
   For $m\geq 3$, the warped cone $\bigsqcup_n (M,\delta_G^{n})$ associated to the action $SL_m(\mathbb{Z}) \curvearrowright \mathbb{R}^m/\mathbb{Z}^m$ has geometric property (T).
\end{thm}
This example reveals a sharp contrast between isometric and non-isometric settings for geometric property (T) of warped cones.

Finally, we show that, for free probability-measure-preserving
Lipschitz actions on compact Riemannian manifolds, geometric
property (T) of the warped cone forces the acting group to have
Kazhdan’s property (T).

\begin{thm}\label{if not kazhdan intro}
Let $G$ be a finitely generated group and $\alpha:G\curvearrowright M$ be a free and probability measure preserving Lipschitz action on a compact Riemannian manifold $M$. If $G$ does not have Kazhdan's property (T), then the warped cone $\bigsqcup (M,\delta_G^{t(n)})$ associated to $G\curvearrowright M$ does not have geometric property (T) for any sequence of scaling parameters $\{t(n)\}_n$ with $\lim t(n) \to \infty$.
\end{thm}

This gives an analogue for warped cones of the ``only if'' direction of \cite[Theorem~7.3]{Willett2014Geometric} about geometric property (T) for box spaces and answers the question \cite[Question~11.6]{Jeroen2021Geometric} in the probability measure preserving case.

\section{Preliminaries on warped cones and geometric property (T)}

We first fix some notation from coarse geometry that will be used throughout
the paper. Let $(X,d)$ be a metric space, where we allow the metric to take
the value $\infty$. A subset $E\subseteq X\times X$ is called
\emph{controlled} if 
\[
    \sup_{(x,y)\in E} d(x,y)<\infty.
\]
For subsets $E,F\subseteq X\times X$, we write
\begin{align*}
    E^{-1}&:=\{(x,y)\in X\times X:(y,x)\in E\},\\
    E\circ F
    &:=\{(x,z)\in X\times X:\exists y\in X \text{ such that } (x,y)\in E,\ (y,z)\in F
    \},
  \end{align*}
and denote by $E^{\circ n}$ the $n$-fold composition of $E$.
For $x\in X$, we also write
\[
    E_x:=\{y\in X:(x,y)\in E\}.
\]
A controlled set $E$ is called \emph{coarse generating} if every controlled
set $F\subseteq X\times X$ is contained in $E^{\circ n}$ for some
$n\in\mathbb N$.

If $(X,d,\mu)$ is a metric measure space and
$T\in B(L^2(X,\mu))$, we define
\[
\operatorname{supp}(T)
:=
\left\{
(x,y)\in X\times X:
\chi_VT\chi_U\neq0
\text{ for every open }U\ni x,\ V\ni y
\right\}.
\]
We denote by $\ControlledSupport[X]$ the algebra of operators
$T\in B(L^2(X,\mu))$ whose support is controlled. When necessary, we write
$\ControlledSupport[(X,d)]$ to specify the metric.
For the bounded geometry spaces considered below,
$\ControlledSupport[X]$ is a $*$-algebra
\cite[Proposition~3.7]{Jeroen2021Geometric}.

\subsection{Warped cones}

In this subsection, we recall the notion of warped cones. Since we are interested in expanders and geometric property $T$, we use the discretized version as in \cite[Section~6]{Federico2019Measure}.

\begin{dfn}\label{warped system}\leavevmode
\rm{
\begin{enumerate}[(i)]
    \item Let $(X,d)$ be a proper metric space and let $G$ be a finitely generated group with a finite symmetric generating set $S\subseteq G$ acting on $X$ by homeomorphisms. We denote by $\ell$ the word-length function associated to $S$. The warped distance $\delta_{G}(x,y)$ between two points $x,y\in X$ is defined to be
    \begin{align*}
       \displaystyle \delta_{G}(x,y):=\inf \sum \left(d(g_ix_i,x_{i+1})+\ell(g_i)\right),
    \end{align*}
    where the infimum is taken over all finite sequences $x=x_0,x_1,\cdots,x_N=y$ in $X$ and $g_0,g_1,\cdots, g_{N-1}$ in $G$.

    \item Let $(M,d)$ be a compact metric space and let $\{t(n)\}_n$ be an increasing sequence of positive numbers such that $\lim_{n\to \infty} t(n)=\infty$. For each $n$, let $d^{t(n)}$ be the rescaled metric $ d^{t(n)}(x,y):=t(n)d(x,y)$. Suppose that a finitely generated group $G$ acts on $M$. We denote by
    $\delta_G^{t(n)}$ the warped metric associated to $d^{t(n)}$ and the $G$-action. The associated warped cone is the disjoint union $\bigsqcup_n (M,\delta_G^{t(n)})$. Throughout the paper, different components of this disjoint union are taken to be at infinite distance from each other.
\end{enumerate}
    }
\end{dfn}

Roe proved \cite[Proposition~1.10]{Roe2005Warped} that the warped cone associated to a Lipschitz action on a compact Riemannian manifold has bounded geometry; in particular, it admits a discretization by graphs of uniformly bounded degree.

We shall use the following coarse generating set.
\begin{lem}[Lemma 11.7, \cite{Jeroen2021Geometric}]\label{Coarse Generating}
\rm{
    Let $M$ be a compact Riemannian manifold with Riemannian distance $d_M$, and let $G$ be a finitely generated group with a finite symmetric generating set $S\subseteq G$ containing the identity. If $G$ acts on $M$ by Lipschitz homeomorphisms, then for every $r>0$, the symmetric controlled set
    \begin{align*}
        E_r=\bigsqcup_n \{(x,y)&\in (M,\delta^{t(n)}_G)\times (M,\delta^{t(n)}_G):\\
        &\exists~ x'\in M, s\in S \text{ s.t. } t(n)d_M(x,x')<r/2 \text{ and } t(n)d_M(sx',y)<r/2\}
    \end{align*}
    is a coarse generating set for the warped cone $\bigsqcup(M,\delta^{t(n)}_G)$.
    }
\end{lem}

\subsection{Laplacians}

In this subsection, we recall the coarse Laplacian $\Delta_{E_r}$, the group Laplacian $\Delta_G$, and the local Laplacian $L_r$ as in
\cite[Section~3]{DengToyotaGeometric}. Let $S\subseteq G$ be the finite symmetric generating set containing the identity. We assume that $S$ is enough large so that $S=S'^3$ for another generating set $S'$ of $G$ (cf. \cite[Proposition~7.9]{Jeroen2021Geometric}). Let $\alpha:G\curvearrowright M$ be a free isometric action on a compact $m$-dimensional Riemannian manifold $M$, and let $\mu$ be the Riemannian measure on $M$. In particular, the action preserves $\mu$.

We denote the warped cone corresponding to the sequence  $\{t(n)\}_n$ by $X=\bigsqcup (X_{t(n)},\delta_G^{t(n)})$. Since the action is isometric, the controlled set $E_r$ introduced in
Lemma~\ref{Coarse Generating} can be written as
\begin{align}
E_r=
\bigsqcup_n\{(x,y)\in X_{t(n)}\times X_{t(n)}:\exists s\in S \text{ s.t. } t(n)d_M(sx,y)<r\}.
\end{align}

We define the coarse Laplacian, the local Laplacian and the group Laplacian
\begin{align*}
    \Delta_{E_r}=\bigoplus_n  \Delta_{E_r,n}&:\bigoplus L^2(X_{t(n)},\mu_{t(n)})\to \bigoplus L^2(X_{t(n)},\mu_{t(n)})\\
   L_r=\bigoplus_n L_{r,n}&:\bigoplus L^2(X_{t(n)},\mu_{t(n)})\rightarrow \bigoplus L^2(X_{t(n)},\mu_{t(n)}),\\
   \Delta_G=\bigoplus_n \Delta_{G,n}&:\bigoplus L^2(X_{t(n)},\mu_{t(n)})\rightarrow \bigoplus L^2(X_{t(n)},\mu_{t(n)})
\end{align*}
by
\begin{align*}
\begin{split}
(\Delta_{E_r,n}\xi)(x)&=\int_{(E_r)_x} (\xi(x)-\xi(y) )d\mu_{t(n)}(y)\\
(L_{r,n}\xi)(x)&=\int_{B_{\frac{r}{t(n)}}(x;d_M)}(\xi(x)-\xi(y))d\mu_{t(n)}(y)\\
(\Delta_{G,n}\xi)(x)&=\sum_{s\in S}(\xi(x)-\xi(sx)).
\end{split}
\end{align*}

By \cite[Proposition~7.9, Lemma~11.7]{Jeroen2021Geometric}, geometric property (T) for warped cones can be characterized in terms of the spectral gap of the coarse Laplacian.
\begin{prop}
    The warped cone $X=\bigsqcup (X_{t(n)},\delta_G^{t(n)})$ has geometric property (T) if and only if there exists $\delta>0$ such that for every unital $*$-homomorphism $\rho:\ControlledSupport[X]\rightarrow B(\HH)$, we have  
    \begin{align*}
        \sigma(\rho(\Delta_{E_r}))\subset \{0\}\cup [\delta,\infty).
    \end{align*}
\end{prop}

For each $n$, let us define a function on $M$ by
$$\phi_n(x):=t(n)^m\cdot\mu(B_{\frac{r}{t(n)}}(x;d_M))$$ 
for any $x\in M$.  
Define a function $\phi \in L^{\infty}(\bigsqcup (X_{t(n)},\mu_{t(n)}))$ by $\phi(x):=\phi_n(x)$ for any $x\in X_{t(n)}$. Then we have the following.

\begin{lem}[\cite{DengToyotaGeometric} Lemma~3.1]\label{local and group laplacians}
With these notations, we have 
\begin{enumerate}
 \item[(1)] $\Delta_{E_r}=|S|\phi-(|S|-\Delta_G)(\phi-L_{r})$ for sufficiently small $r>0$;
 \item[(2)] the function $\phi$ and the local Laplacian $L_r$ are $G$-equivariant.
\end{enumerate}
\end{lem}

\section{Proof of the Theorem~\ref{main}}

We continue to assume that $\alpha:G\curvearrowright M$ is free and isometric, and we use the notation introduced in the previous section. We may further assume that the action is ergodic. Indeed, otherwise the associated warped cone is not an expander by \cite{Federico2019Measure}, and hence cannot have
geometric property (T). We denote by $X=\bigsqcup(X_{t(n)},\delta^{t(n)}_G)$ and $Y=\bigsqcup(X_{t(n)},d^{t(n)})$, the coarse disjoint unions of $M$ with warped or non-warped distance, respectively (Definition~\ref{warped system}).
By \cite[Lemma~11.8]{Jeroen2021Geometric}, there is a 
$*$-homomorphism from the algebraic crossed product
   \begin{align*}
    \Psi:\ControlledSupport[Y]\rtimes_{\alg}G \rightarrow \ControlledSupport[X],
   \end{align*}
   which fits into the commutative diagram
\begin{center}
\begin{tikzcd}
\ControlledSupport[Y]\rtimes_{\alg}G \arrow[r, "\Psi"] \arrow[d, "q\otimes 1"]
& \ControlledSupport[X] \arrow[d, "q" ] \\
\left(\ControlledSupport[Y]/I\right)\rtimes_{\alg}G \arrow[r,"\cong"]
&\ControlledSupport[X]/I,
\end{tikzcd}
\end{center}
Here
\[
I:=\bigoplus_n B(L^2(M,\mu_{t(n)}))
\]
is the direct sum. The right vertical map is the quotient map by $\overline I$, while the left vertical map is induced by the $G$-equivariant quotient. The bottom horizontal map is an isomorphism, if the action $G\curvearrowright M$ is free. (This is true for all free Lipschitz actions on a compact Riemannian manifold without assuming that the action is isometric.)
Let 
$$\widetilde{\Delta}_{E_r}:=|S|\phi-(\phi-L_r)\otimes(|S|-\Delta_G)\in \overline{\ControlledSupport[Y]}\rtimes G.$$
Since $\Psi(\widetilde{\Delta}_{E_r})=\Delta_{E_r}$, to prove the main theorem, it suffices to show that $(q\otimes 1)(\widetilde{\Delta}_{E_r})$ does not have a spectral gap in $$\left(\overline{\ControlledSupport[Y]}^{L^2}/\overline{I}\right)\rtimes_{\max}G\cong\left(\overline{\ControlledSupport[Y]}^{L^2}\rtimes_{\max}G\right)/(\overline{I}\rtimes_{\max}G).$$
Here $\overline{\ControlledSupport[Y]}^{L^2}$ is the norm completion of $\ControlledSupport[Y]$ in $B(L^2(Y))$.

We shall construct a covariant system $(\pi,U,\HH)$ for the
$C^*$-dynamical system
\[
\ad:G\curvearrowright\overline{\ControlledSupport[Y]}^{L^2}
\]
which detects the failure of a spectral gap for
$\widetilde{\Delta}_{E_r}$. For such a covariant system, we denote by
$\HH_G$ the $G$-invariant subspace of $\HH$. 
\begin{prop}\label{sufficient}
To prove Theorem~\ref{main}, it suffices to construct a covariant system
$(\pi,U,\HH)$ for $\ad:G\curvearrowright\overline{\ControlledSupport[Y]}^{L^2}$
satisfying the following conditions:
    \begin{enumerate}[i)]
        \item $(\pi,U,\HH)$ is a direct sum of a sequence of covariant systems $(\pi^{(n)},U^{(n)},\HH^{(n)})_n$ such that $\HH^{(n)}$ is the support of $B(L^2(X_{t(n)},\mu_{t(n)}))\subseteq \overline{\ControlledSupport[Y]}^{L^2}$, i.e. $\HH^{(n)}:=\pi(1 |_{L^2(X_{t(n)},\mu_{t(n)})})\HH$.

        \item $\pi(L_r)|_{\HH_G}$ does not have a spectral gap in the quotient $B(\HH_G)/\oplus_n B(\HH^{(n)}_G)$.
    \end{enumerate}
\end{prop}

\begin{proof}
We denote the associated representation of the crossed product $\overline{\ControlledSupport[Y]}^{L^2}\rtimes G$ by $\widetilde{\pi}$. Since we have a quotient
    \begin{align*}
        \left(\overline{\ControlledSupport[Y]}^{L^2}\rtimes_{\max}G\right)/(\overline{I}\rtimes_{\max}G) \twoheadrightarrow \left(\overline{\ControlledSupport[Y]}^{L^2}\rtimes_{\widetilde{\pi}}G\right)/(\overline{I}\rtimes_{\widetilde{\pi}}G),
    \end{align*}
    it suffices to show that $\widetilde{\pi}(\widetilde{\Delta}_{E_r})$ does not have a spectral gap in the latter quotient.
    
    On $\HH_G$, we have $\widetilde{\pi}(\widetilde{\Delta}_{E_r})=|S|\pi(L_r)$. Therefore, the second condition implies that $\widetilde{\pi}(\widetilde{\Delta}_{E_r})$ does not have spectral gap in $B(\HH)/\oplus_n B(\HH^{(n)})$. By the first condition, we have that $\overline{I}\rtimes_{\widetilde{\pi}} G \subseteq \oplus_n B(\HH^{(n)})$ and therefore there is a unital $*$-homomorphism
    \begin{align*}
        \left(\overline{\ControlledSupport[Y]}^{L^2}\rtimes_{\widetilde{\pi}}G\right)/(\overline{I}\rtimes_{\widetilde{\pi}}G) \to B(\HH)/\oplus_n B(\HH^{(n)}),
    \end{align*}
    So $\widetilde{\pi}(\widetilde{\Delta}_{E_r})$ does not have a spectral gap also  in $\left(\overline{\ControlledSupport[Y]}^{L^2}\rtimes_{\widetilde{\pi}}G\right)/(\overline{I}\rtimes_{\widetilde{\pi}}G)$. 
\end{proof}

Now we recall the covariant system $(\pi,U,\HH)$ of the $C^*$-dynamical system $G\curvearrowright \overline{\ControlledSupport[Y]}^{L^2}$ constructed in \cite{DengToyotaGeometric}.
Let $\omega_n$ be the trace on $B(L^2(X_{t(n)},\mu_{t(n)}))$ and we denote by
$$\HH^{(n)}:=\{T\in L^2(X_{t(n)}\times X_{t(n)},\mu_{t(n)}\times \mu_{t(n)}):\omega_n(T^*T)<\infty\}$$ the Hilbert space consists of Hilbert-Schmidt kernels.
For kernel operators $k,k'\in \HH^{(n)}$, we have
$$\omega_n(k^*k')=\int_{M\times M}\overline{k(x,y)}k'(x,y)d(\mu_{t(n)}\times\mu_{t(n)})(x,y).$$
The group $G$ acts on $\HH^{(n)}$ by conjugation; we denote this action by $U^{(n)}$. Explicitly,
\[
(U_g^{(n)}k)(x,y)=k(g^{-1}x,g^{-1}y).
\]
Moreover, $\overline{\ControlledSupport[Y]}^{L^2}$ acts on $\HH^{(n)}$ by left multiplication; we denote this representation by $\pi^{(n)}$. For each $n$, let
\[
\HH_G^{(n)}:=\left\{k\in\HH^{(n)}:k(gx,gy)=k(x,y)\text{ for all }x,y\in X_{t(n)},\ g\in G \right\}
\]
and denote the $G$-invariant subspace of $\HH$ by $\HH_G=\bigoplus_n\HH_G^{(n)}$.

In \cite[Lemma~4.2]{DengToyotaGeometric}, Deng and the author proved that, when $M$ is a compact Lie group and $G\subseteq M$ is a discrete dense subgroup acting by left multiplication, the restriction of $\pi^{(n)}(L_r)$ to the $G$-invariant Hilbert--Schmidt kernels is unitarily equivalent to $L_{r,n}$ on $L^2(X_n,\mu_{t(n)})$. For a general isometric action, this identification has to be modified. Let $K=\overline{\alpha(G)}$, where the closure is taken in $\Isom(M)$.

\begin{lem}
   If the action $\alpha:G\curvearrowright M$ is isometric and ergodic, then $K$ acts transitively on $M$.
\end{lem}

\begin{proof}
    Fix $x_0\in M$ and consider the continuous $G$-invariant function $x\mapsto d_M(x, Kx_0)$. Since $K$ is a compact subgroup of $\Isom(M)$, we have that $d_M(x, Kx_0)=0$ if and only if $x\in Kx_0$. Therefore since the action is ergodic, this function must be constant,  which implies that $Kx_0=M$ and the action is transitive.
\end{proof}

Fix $x_0\in M$ and denote the stabilizer subgroup of $x_0$ by
\begin{align*}
    H:=K_{x_0}=\{g\in K:gx_0=x_0 \}.
\end{align*}
Now, we have the generalization of \cite[Lemma~4.2]{DengToyotaGeometric}. We denote by $L^2_K(X_{t(n)}\times X_{t(n)})$ (resp. $L^2_{H}(X_{t(n)},\mu_{t(n)})$) the $K$-invariant (resp. $H$-invariant) subspace of $L^2(X_{t(n)}\times X_{t(n)})$ (resp. $L^2(X_{t(n)},\mu_{t(n)})$). Note that since $K$ is the closure of $\alpha(G)$, we have $L^2_K(X_{t(n)}\times X_{t(n)})=\HH_G^{(n)}$.

\begin{lem}\label{intertwine}
    Let $M$ be a compact Riemannian manifold and $K\subset \Isom(M)$ be a compact subgroup which acts on $M$ transitively. Then for any fixed $x_0\in M$ the map $k\mapsto \sqrt{\mu_{t(n)}(M)} k(\cdot,x_0)$ defines an unitary
    \begin{align*}
        W:L^2_K(X_{t(n)}\times X_{t(n)}) \rightarrow L^2_{H}(X_{t(n)},\mu_{t(n)})
    \end{align*}
    such that 
    \begin{equation}\label{intertwine formula}
        \pi^{(n)}(L_r)|_{L^2_K(X_{t(n)}\times X_{t(n)})}=W^*\left(L_{r,n}|_{L^2_{H}(X_{t(n)},\mu_{t(n)})}\right)W.
    \end{equation}
    
\end{lem}

\begin{proof}
    First, note that since $K$  (resp. $H$) are compact, by averaging in terms of the Haar measure, any element $L^2_K(X_{t(n)}\times X_{t(n)})$ (resp. $L^2_{H}(X_{t(n)},\mu_{t(n)})$) admits an approximation by $K$-invariant continuous functions on $X_{t(n)}\times X_{t(n)}$ (resp. $H$-invariant continuous functions on $X_{t(n)}$).
    For a $K$-invariant continuous function $k\in C(X_{t(n)}\times X_{t(n)})$, define $Wk \in C(X_{t(n)})$ by
    \begin{align*}
        (Wk)(x):=\sqrt{\mu_{t(n)}(M)}k(x,x_0).
    \end{align*}
    Then first, $Wk$ is $H$-invariant since
    \begin{align*}
        (Wk)(hx)=\sqrt{\mu_{t(n)}(M)}k(hx,x_0)=\sqrt{\mu_{t(n)}(M)}k(x,h^{-1}x_0)=\sqrt{\mu_{t(n)}(M)}k(x,x_0)
    \end{align*}for $h\in H$. Moreover since the action is transitive, for any $y\in X_{t(n)}$, there exists $a_y\in K$ such that $y=a_yx_0$ and so
    \begin{align*}
        \int_M |k(x,y)|^2 d\mu_{t(n)}(x) &= \int_M |k(a_y^{-1}x,x_0)|^2 d\mu_{t(n)}(x)\\
        &=\int_M |k(z,x_0)|^2 d\mu_{t(n)}(z)
    \end{align*}
    is independent of $y$. Therefore we have
    \begin{align*}
        \int_M \int_M |k(x,y)|^2 d\mu_{t(n)}(x)d\mu_{t(n)}(y)=\mu_{t(n)}(M) \int_M |k(x,x_0)|^2 d\mu_{t(n)}(x)=\|Wk\|^2_{L^2(X_{t(n)},\mu_{t(n)})}.
    \end{align*}
    Thus $W$ can be extended to an isometry 
    \begin{equation}
        W:L^2_K(X_{t(n)}\times X_{t(n)}) \rightarrow L^2_{H}(X_{t(n)},\mu_{t(n)}).
    \end{equation}
    We show that it is surjective. Let $f\in C(X_{t(n)})$ be an $H$-invariant function.  We define $k\in C(X_{t(n)} \times K)$ by $k(x,g):=\frac{1}{\sqrt{\mu_{t(n)}(M)}}f(g^{-1}x)$. Then we have $$k(gx,gg')=\frac{1}{\sqrt{\mu_{t(n)}(M)}}f((gg')^{-1}gx)=\frac{1}{\sqrt{\mu_{t(n)}(M)}}f((g')^{-1}x)=k(x,g')$$ for all $x\in X_{t(n)}$ and $g,g'\in K$, therefore $k$ is $K$-invariant. Moreover, for $h\in H$, we have $k(x,gh)=\frac{1}{\sqrt{\mu_{t(n)}(M)}}f(h^{-1}g^{-1}x)=k(x,g)$, since $f$ is $H$-invariant. Therefore, $k$-descends to a continuous function $\widetilde{k}$ on
    \begin{align*}
        X_{t(n)}\times (K/H)\cong X_{t(n)}\times X_{t(n)}; (x,gH)\mapsto (x,gx_0).
    \end{align*}
     Then by construction, we have
    \begin{align*}
        (W\widetilde{k})(x)=\sqrt{\mu_{t(n)}(M)}\widetilde{k}(x,x_0)=f(x)
    \end{align*}
    So we have $W\widetilde{k}=f$ and $W$ is surjective. 

    It remains to prove \eqref{intertwine formula}. For $k,k'\in L^2_K(X_{t(n)}\times X_{t(n)})$, we have
    \begin{align*}
    &\omega_n\left(\left({k'}^*(\pi^{(n)}(L_r) k\right)\right)\\
    =&\int_M\left(\int_M\overline{k'(x,y)}\phi(x)k(x,y)d\mu_{t(n)}(x)\right)d\mu_{t(n)}(y)\\
    &-\int_{M}\left(\int_M \overline{k'(x,y)}\left(\int_{M}\chi_{B_{\frac{r}{t(n)}(x)}}(z)k(z,y)d\mu_{t(n)}(z)\right)d\mu_{t(n)}(x)\right)d\mu_{t(n)}(y)\\
    \end{align*}
    By the change of variables $x\mapsto a_yx$, this is equal to
\begin{align*}&\int_M\left(\int_M\overline{k'(a_y^{-1}x,x_0)}\phi(x)k(a_y^{-1}x,x_0)d\mu_{t(n)}(x)\right)d\mu_{t(n)}(y)\\
    &-\int_{M}\left(\int_M \overline{k'(a_y^{-1}x,x_0)}\left(\int_{M}\chi_{B_{\frac{r}{t(n)}(x)}}(z)k(a_y^{-1}z,x_0)d\mu_{t(n)}(z)\right)d\mu_{t(n)}(x)\right)d\mu_{t(n)}(y)\\
    =&\int_M\left(\int_M\overline{k'(a_y^{-1}x,x_0)}\phi(x)k(a_y^{-1}x,x_0)d\mu_{t(n)}(x)\right)d\mu_{t(n)}(y)\\
    &-\int_{M}\left(\int_M \overline{k'(a_y^{-1}x,x_0)}\left(\int_{M}\chi_{B_{\frac{r}{t(n)}(a_y^{-1}x)}}(z)k(z,x_0)d\mu_{t(n)}(z)\right)d\mu_{t(n)}(x)\right)d\mu_{t(n)}(y)\\
    =&\mu_{t(n)}(M)\left(\int_M \overline{k'(x,x_0)}\phi(x)k(x,x_0)d\mu_{t(n)}(x)\right)\\
    &-\mu_{t(n)}(M)\left(\int_M \overline{k'(x,x_0)}\left(\int_{M}\chi_{B_{\frac{r}{t(n)}(x)}}(z)k(z,x_0)d\mu_{t(n)}(z)\right)d\mu_{t(n)}(x)\right)\\
   = &\left\langle \sqrt{\mu_{t(n)}(M)}k'(\cdot,x_0), (L_r)(\sqrt{\mu_{t(n)}(M)}k(\cdot,x_0)) \right\rangle_{L^2(M,\mu_{t(n)})}=\langle Wk',L_r(Wk) \rangle_{L^2(M,\mu_{t(n)})}
\end{align*}
This proves \eqref{intertwine formula}.
\end{proof}

To verify condition~(ii) of Proposition~\ref{sufficient}, it remains to construct almost $L_r$-invariant vectors in the $H$-invariant subspace $\bigoplus_nL_H^2(X_{t(n)},\mu_{t(n)})$. In \cite{DengToyotaGeometric}, Deng and the author constructed such vectors on the full $L^2$-space by comparing $L_r$ with the heat operators $\left(1-\exp\left(-\frac{\Delta_M}{t(n)^2}\right)\right)_n$. The following lemma shows that the same spectral information remains available in the $H$-invariant subspace. In particular,
\[
\sigma(\Delta_M|_{L_H^2(M)})
=
\sigma(\Delta_M).
\]
\begin{lem}\label{equivariant kernel}
    Suppose that $M$ is a compact Riemannian manifold and $K$ is a compact subgroup of $\Isom(M)$ that acts on $M$ transitively. Fix $x_0\in M$ and denote by $H:=K_{x_0}$ the stabilizer subgroup of $x_0$. Let $\Delta_M$ be the scalar Laplace--de Rham operator on $M$, i.e. the restriction of Hodge Laplacian to zero-forms (smooth functions). If $\lambda$ is an eigenvalue of $\Delta_M$, then we have that $\ker(\Delta_M-\lambda) \cap L^2_H(M)\neq \{0\}$. 
\end{lem}

\begin{proof}
    We denote by $P_{\lambda}$ the spectral projection onto $E_{\lambda}$.
    Let $\xi_1,\cdots,\xi_n$ be an orthonormal basis of $E_{\lambda}$. Then for any $\eta\in L^2(M)$, we have
    \begin{align*}
        (P_{\lambda}\eta)(x)=\sum_{i=1}^n \left(\int_M \overline{\xi_i(y)}\eta(y)d\mu (y)\right) \xi_i(x)&=\int_M \left(\sum_{i=1}^n\xi_i(x)\overline{\xi_i(y)}\right)\eta(y)d\mu (y)\\&=\int_M p_{\lambda}(x,y)\eta(y)d\mu (y)
    \end{align*}
    by denoting $p_{\lambda}(x,y):=\sum_{i=1}^n\xi_i(x)\overline{\xi_i(y)}$. Since $K$ acts on $M$ by isometries, $p_{\lambda}(gx,gy)=p_{\lambda}(x,y)$ for all $x,y\in M$ and $g\in K$. Therefore by transitivity, the function $x \mapsto p_{\lambda}(x,x)$ is constant and also
    \begin{align*}
        \int_M p_{\lambda}(x,x) d\mu(x)=\int_M \sum_{i=1}^n\xi_i(x)\overline{\xi_i(x)} d\mu(x) =\dim E_{\lambda}.
    \end{align*}
    So we have that the continuous function $\xi:=p_{\lambda}(\cdot,x_0)$ is non-zero. Now, it suffices to prove the following (i) $\xi$ is $H$-invariant, and (ii) $\xi\in E_{\lambda}$. For the first claim, 
    \begin{align*}
        \xi(hx)=p_{\lambda}(hx,x_0)=p_{\lambda}(x,h^{-1}x_0)=p_{\lambda}(x,x_0)
    \end{align*}
    for all $x\in M$ and $h\in H$ by the $K$-invariance of $p_{\lambda}$.
    The second claim follows since $\xi=\sum \overline{\xi_i}(x_0) \xi_i\in E_{\lambda}$.
\end{proof}

Before proving Theorem~\ref{main}, we need the following comparison of local Laplacians with different parameter, to compare local Laplacian and de-Rham Laplacian in the restricted quotient $B(\oplus L^2_{H}(X_{t(n)},\mu_{t(n)}))/\oplus B(L^2_{H}(X_{t(n)},\mu_{t(n)}))$.
\begin{lem}
    Let $0<r<R$ and let $N\in\mathbb{N}$ be the smallest integer such that $Nr>R$. Then there exists $C_{r,R}>0$ such that 
    \begin{align}\label{comparison local laplacians}
        L_{R,n}\leq C_{r,R} L_{r,n}
    \end{align}
     for every $n$.
\end{lem}

\begin{proof}
    Let $k_{r,n}(x,y):=t(n)^m\chi_{B_{\frac{r}{t(n)}}(x;d_M)}(y)$. First, we prove $k_{R,n}\leq Ck_{r,n}^{*N}$ for some $C>0$ independent of $n$. Fix $\delta>0$ such that $(r-2\delta)N>R$. Fix $x,y\in M$ with $d^{t(n)}(x,y)\leq R$ and then we can take points $x=p_0,p_1,\cdots, p_N=y\in M$ such that $d^{t(n)}(p_i,p_i+1)\leq r-2\delta$. Note that $\mu_{t(n)}(B_{\frac{\delta}{t(n)}}(p;d_M))$ is uniformly bounded below by some $\varepsilon>0$ independently of $p\in M$ and $n$. Now we have that 
    \begin{align*}
        k_{r,n}^{*N}(x,y)&=\int_M\cdots \int_M \chi_{B_{\frac{r}{t(n)}}(x;d_M)}(z_1)\cdots \chi_{B_{\frac{r}{t(n)}}(z_{N-1};d_M)}(y) d\mu_{t(n)}(z_{1})\cdots d\mu_{t(n)}(z_{N-1}) \\
        &\geq \mu_{t(n)}(B_{\frac{\delta}{t(n)}}(p_1;d_M))\cdot \cdots \cdot \mu_{t(n)}(B_{\frac{\delta}{t(n)}}(p_{N-1};d_M))\geq \varepsilon^{N-1}.
    \end{align*}
    Since if $d^{t(n)}(x,y)> R$, then we have $k_{R,n}(x,y)=0$, we have $k_{R,n}\leq Ck_{r,n}^{*N}$ for $C=\frac{1}{\varepsilon^{N-1}}$.

    Therefore, $L_{R,n}$ is dominated by the operator $L_{k^{*N}}$ defined by
\begin{align*}
    (L_{k^{*N}}f)(x):=\int_M k^{*N}(x,y)(f(x)-f(y))d\mu(y).
\end{align*}  
Since for isometric transitive action, the function
$$\phi_n(x):=\mu_{t(n)}(B_{\frac{r}{t(n)}}(x;d_M))=t(n)^m\cdot\mu(B_{\frac{r}{t(n)}}(x;d_M))$$ 
is constant, denoted by just $\phi_n$. If we denote by $A_{r,n}$ the kernel operator corresponding to $k_{r,n}$, then we have that
\begin{align*}
    &(L_{k^{*N}}f)(x)\\=&\int_M\int_M\cdots \int_M k(x,z_1)\cdots k(z_{n-1},y)f(x) d\mu(z_1)\cdots d\mu(z_{n-1})d\mu(y)\\
    &-\int_M\int_M\cdots \int_M k(x,z_1)\cdots k(z_{n-1},y)f(y) d\mu(z_1)\cdots d\mu(z_{n-1})d\mu(y)\\
    =&\int_M f(x) \left(\int_M k(x,z_1)d\mu(z_1)\right) \left(\int_M k(z_1,z_2)d\mu(z_2)\right)\cdots \left(\int_M k(z_{n-1},y)d\mu(y)\right)\\
    &-\int_M k(x,z_1)\cdots\left(k(z_{n-2},z_{n-1})\left(\int_M k(z_{n-1},y)f(y)d\mu(y)\right)d\mu(z_{n-1})\right)\cdots d\mu(z_1)\\
    =&\phi_n^Nf(x)-(A_{r,n}^Nf)(x).
\end{align*}
Therefore, we have
\begin{align*}
    L_{k^{*N}}=\phi_n^N-A_{r,n}^{N}=\phi_n^{N-1}\left(\phi_n-A_{r,n}\right)\left(1+\frac{A_{r,n}}{\phi_n}+\cdots +\left(\frac{A_{r,n}}{\phi_n}\right)^{N-1}\right) \leq N\phi_n^{N-1}L_{r,n}.
\end{align*}
and since $\phi_n$ is uniformly bounded, we have the desired inequality \eqref{comparison local laplacians}.
\end{proof}

\begin{proof}[Proof of the Theorem~\ref{main}]
By Lemma~\ref{intertwine}, to prove that $\pi(L_r)|_{\HH_G}$ does not have spectral gap in the quotient $B(\HH_G)/\oplus_n B(\HH^{(n)}_G)$, it suffices to show that $$\left(L_{r,n}|_{L^2_{H}(X_{t(n)},\mu_{t(n)})}\right)_n$$ does not have spectral gap in $B(\oplus L^2_{H}(X_{t(n)},\mu_{t(n)}))/\oplus B(L^2_{H}(X_{t(n)},\mu_{t(n)}))$. Since $H\subseteq K$ acts on $M$ by isometries, $\Delta_M$ can be restricted to $L^2_{H}(M)$. By \cite[Lemma~3.3]{DengToyotaGeometric} there exist $ C,D>0$ such that for every $\varepsilon>0$ there exists $R>0$ with
\begin{align}\label{local-deRham}
     0\leq L_{r,n}\leq C\left(1-\exp{\left(-\frac{\Delta_M}{t(n)^2}\right)}\right) \leq DL_{R,n}+\varepsilon.
\end{align}
Inequalities \eqref{local-deRham} and \eqref{comparison local laplacians} implies that $L_{r,n}$ is dominated by $1-\exp{\left(-\frac{\Delta_M}{t(n)^2}\right)}$ and they have the same kernel in  the quotient $B(\oplus L^2_{H}(X_{t(n)},\mu_{t(n)}))/\oplus B(L^2_{H}(X_{t(n)},\mu_{t(n)}))$ (after faithfully represented on a Hilbert space). Therefore, it suffices to show that $$\left(1-\exp{\left(-\frac{\Delta_M|_{L^2_{H}(M)}}{t(n)^2}\right)}\right)_n$$ does not have spectral gap in $B(\oplus L^2_{H}(M)/\oplus B(L^2_{H}(M))$. In the proof of \cite[Proposition~3.2]{DengToyotaGeometric}, using the Weyl's law (i.e.
\begin{align*}
    \frac{\# \{\text{eigenvalues of }\Delta_M \text{ less than }\Lambda\}}{\Lambda^{m/2} }\to C
\end{align*}
for some $C$),
it was shown that for every $\varepsilon>0$ the union of spectra
\begin{align*}
   \bigcup_n \sigma\left(1-\exp{\left(-\frac{\Delta_M}{t(n)^2}\right)}\right)
\end{align*}
has an accumulation point $\lambda$ in $(0,\varepsilon)$. By Lemma~\ref{equivariant kernel}, $\lambda$ is also an accumulation point of
\begin{align*}
   \bigcup_n \sigma\left(1-\exp{\left(-\frac{\Delta_M|_{L^2_H(M)}}{t(n)^2}\right)}\right).
\end{align*}
This implies that $\lambda$ is in the spectrum of $\left(1-\exp{\left(-\frac{\Delta_M|_{L^2_H(M)}}{t(n)^2}\right)}\right)_n$ in the quotient $B(\oplus L^2_{H}(M))/\oplus B(L^2_{H}(M))$.
\end{proof}

\section{Warped cones for non-isometric p.m.p actions}

In this section, we study geometric property (T) for warped cones associated to not necessarily isometric but probablity measure preserving actions on compact Riemannian manifolds. In the first subsection, we prove that Theorem~\ref{main} can not be extended to non-isometric actions. Then, in the second subsection, we prove that for free probability measure preserving actions, if $G$ does not have Kazhdan's property (T), then the warped cone does not have geometric property (T)

\subsection{Example of an ergodic action whose warped cone has geometric property (T)}

In this subsection, we show that Theorem~\ref{main} does not extend to general non-isometric actions. In particular, we prove the following:
\begin{thm}\label{positive example}
   For $m\geq 3$, the warped cone $\bigsqcup_n (M,\delta_G^{n})$ associated to the action $SL_m(\mathbb{Z}) \curvearrowright \mathbb{R}^m/\mathbb{Z}^m$ has geometric property (T).
\end{thm}

Let $G:=SL_m(\mathbb{Z})$ with a fixed finite symmetric generating set $S\ni I_m$, and let $M:=\mathbb{R}^m/\mathbb{Z}^m$
be the flat torus. The group $G$ acts on $M$ by
\[
A\cdot(v+\mathbb{Z}^m):=Av+\mathbb{Z}^m.
\]
We consider the sequence of scale parameters $t(n)=n$.
Then the rescaled metric space $(Y_n,d^n)$ is isometric to
$\mathbb{R}^m/(n\mathbb{Z}^m)$ equipped with the Euclidean metric.
We take the lattice $Z_n:=\mathbb{Z}^m/(n\mathbb{Z}^m)\subseteq Y_n$, which is invariant under the $G$-action. 

The semidirect product $H:=\mathbb{Z}^m\rtimes SL_m(\mathbb{Z})$ acts on $Z_n$ by 
\begin{align*}
    (v,A):x\mapsto Ax+v.
\end{align*}
Indeed, the composition of the action corresponds to the product of the semidirect product, $(v,A)\circ(w,B)x=(v+Aw,AB)x=(v,A)(w,B)x$. We denote by $\widetilde{S}:=\{(\pm e_i,I_m):i=1,\cdots m\}\cup \{(0,A):A\in S\}$ a generating set of $H$. The action $H\curvearrowright Z_n$ defines the Schreier graph
$\sch(H\curvearrowright Z_n)$: its vertex set is $Z_n$, and two vertices
$x,y\in Z_n$ are adjacent if $y=sx$ for some $s\in\widetilde S$.  We denote the associated graph distance on $Z_n$ by $d_H$. Note that the action $H\curvearrowright Z_n$ is transitive, and so the graph $Z_n$ is connected and $|Z_n|\to \infty$. The two types of edges in this Schreier graph correspond to the two types
of moves defining the warped metric. An edge $x\sim Ax$, for $A\in S$,
corresponds to a move by the group action, while an edge
$x\sim x\pm e_i$ corresponds to the local move $x\mapsto x\pm e_i$.
If $Z_n$ is equipped with the $\ell^1$-metric, then the warped metric on $Z_n$ agrees with the Schreier metric $d_H$. Moreover, since every point of $Y_n$ lies within a distance $\frac{m}{2}$ of $Z_n$, the inclusion
\[
(Z_n,\delta_G^n)\hookrightarrow (Y_n,\delta_G^n)
\]
is a quasi-isometry, uniformly in $n$.
 Since the $\ell^1$-metric and the Euclidean metric on $\mathbb{R}^m$
are bi-Lipschitz equivalent with constants depending only on $m$,
replacing the $\ell^1$-metric by the Euclidean metric does not change
the quasi-isometry type of the warped cone.

Thus we obtain the following

\begin{lem}
    The warped cone $\bigsqcup_n (M, \delta_G^n)$ associated to the action $G\curvearrowright M$ is quasi-isometric to the coarse disjoint union of the Schreier graphs $\bigsqcup_n (Z_n,d_H)$.
\end{lem}

Therefore, to prove Theorem~\ref{positive example}, it suffices to show that the coarse disjoint union of the Schreier graphs $\bigsqcup_n (Z_n,d_H)$
has geometric property (T).

Since $H$ acts on each $Z_n$, it induces a unitary representation
\[
\lambda_n:H\to \mathcal U(\ell^2(Z_n)).
\]
We denote by $\lambda:=\bigoplus_n\lambda_n$
the resulting unitary representation of $H$ on $\bigoplus_n\ell^2(Z_n)$.  This representation extends linearly to a unital $*$-homomorphism,
still denoted by $\lambda$,
\begin{align}\label{embedding}
   \lambda: \mathbb{C}[H]\to \mathbb{C}_u\left[\bigsqcup (Z_n,d_H)\right]
\end{align}
Let
\[
\Delta_H
:=
|\widetilde S|-\sum_{s\in\widetilde S}s
\in\mathbb C[H]
\]
be the group Laplacian associated to $\widetilde S$. Its image $\lambda(\Delta_H)$
is the Laplacian of the coarse disjoint union of the corresponding Schreier multigraphs. 
For $x,y\in Z_n$, define the edge multiplicity by
\[
m(x,y):=
\#\{s\in\widetilde S:sx=y\}.
\] Then we have
\begin{align*}
    \lambda(\Delta_H)=\frac{1}{2}\sum_{(x,y)} m(x,y) (e_{x,x}-e_{x,y}-e_{y,x}+e_{y,y}).
\end{align*}

On the other hand, let $\Delta$ denote the Laplacian of the underlying simple graph on $\bigsqcup_n (Z_n,d_H)$. Then
\[
\Delta
=
\frac12
\sum_{\substack{(x,y)\\m(x,y)>0}}
\bigl(
e_{x,x}-e_{x,y}-e_{y,x}+e_{y,y}
\bigr).
\]

\begin{lem}\label{lem:multigraph-simple}
    For any $*$-homomorphism $\pi$ of $\mathbb{C}_u\left[\bigsqcup (Z_n,d_H)\right]$, we have that
\begin{align}\label{multiple edge}
   \pi( \Delta)\leq \pi(\lambda(\Delta_H))\leq |\widetilde{S}|\pi(\Delta).
\end{align}

\end{lem}
\begin{proof}
For each $k=1,2,\cdots |\widetilde{S}|$, we define
\begin{align*}
    E_k:=\bigsqcup_n\{(x,y)\in Z_n\times Z_n:x\neq y, \; m(x,y)=k\},
\end{align*}
By \cite[Lemma~2.7]{Willett2014Geometric}, we can decompose $E_k$ as
\[
E_k
=
E_k^{(1)}\sqcup \cdots \sqcup E_k^{(N)}
\]
 for some $N$ such that, for each $i=1,\cdots,N$, there is a bijection
\[
t_k^{(i)}:A_k^{(i)}\to B_k^{(i)}
\]
with $A_k^{(i)}\cap B_k^{(i)}=\emptyset$ and
\[
E_k^{(i)}
=
\{(x,t_k^{(i)}(x)):x\in A_k^{(i)}\}
\sqcup
\{(t_k^{(i)}(x),x):x\in A_k^{(i)}\}.
\]
Let $v_k^{(i)}\in\mathbb C_u\left[\bigsqcup_n(Z_n,d_H)\right]$ be the partial isometry associated to $t_k^{(i)}$.
Then we have that
\begin{align*}
    \lambda(\Delta_H)-\Delta&=\frac{1}{2}
\sum_{\substack{(x,y)\\x\neq y,\; m(x,y)>0}} (m(x,y)-1)(e_{x,x}-e_{x,y}-e_{y,x}+e_{y,y})\\
    &=\frac{1}{2}\sum_{k=1}^{|\widetilde{S}|}
\sum_{\substack{(x,y)\\x\neq y,\;m(x,y)=k}} (k-1) (e_{x,x}-e_{x,y}-e_{y,x}+e_{y,y})\\
    &=\sum_{k=1}^{|\widetilde{S}|} \sum_{i=1}^N (k-1)(v_k^{(i)}(v_k^{(i)})^*-v_k^{(i)})^*(v_k^{(i)}(v_k^{(i)})^*-v_k^{(i)}).
\end{align*}
Therefore, for any $*$-homomorphism $\pi$ of $\mathbb{C}_u\left[\bigsqcup (Z_n,d_H)\right]$, we have 
\begin{align*}
    \pi(\lambda(\Delta_H)-\Delta)=\sum_{k=1}^{|\widetilde{S}|} \sum_{i=1}^N (k-1)\pi((v_k^{(i)}(v_k^{(i)})^*-v_k^{(i)}))^*\pi((v_k^{(i)}(v_k^{(i)})^*-v_k^{(i)}))\geq 0.
\end{align*}
Similarly, $|\widetilde S|\pi(\Delta)-\pi(\lambda(\Delta_H))\geq0$.
This proves \eqref{multiple edge}.
\end{proof}

Now, we complete the proof of Theorem~\ref{positive example},

\begin{proof}[Proof of Theorem~\ref{positive example}]
    By Lemma~\ref{lem:multigraph-simple}, it suffices to show that there exists $\delta>0$ such that
$$\sigma(\pi(\lambda(\Delta_H)))\subseteq \{0\}\cup [\delta,\infty)$$ for every unital $*$-homomorphism $\pi:\mathbb C_u\left[\bigsqcup_n (Z_n,d_H)\right]\to B(\mathcal H)$. 
Since $H$ has Kazhdan's property (T) for $m\geq 3$ \cite[Example~1.7.4]{KazhdanPropertyT}, 
there exists $\delta>0$ such that for every unital $*$-homomorphism $\rho:\mathbb C[H]\to B(\mathcal H)$, we have $\sigma(\rho(\Delta_H))\subseteq \{0\}\cup [\delta,\infty)$.
The composition
\[
\pi\circ \lambda:\mathbb C[H]\to B(\mathcal H)
\]
is a  unital $*$-homomorphism of $\mathbb{C}[H]$ and so we have that $\sigma(\pi(\lambda(\Delta_H)))\subseteq \{0\}\cup [\delta,\infty)$.
\end{proof}

\subsection{The case where $G$ does not have Kazhdan's property (T)}

In this subsection, we prove the following:

\begin{thm}\label{if not kazhdan}
    Let $G$ be a finitely generated group and $\alpha:G\curvearrowright M$ be a free and probability measure preserving Lipschitz action on a compact Riemannian manifold $M$. If $G$ does not have Kazhdan's property (T), then the warped cone $\bigsqcup (M,\delta_G^{t(n)})$ associated to $G\curvearrowright M$ does not have geometric property (T) for any sequence of scaling parameters $\{t(n)\}_n$ with $\lim t(n) \to \infty$.
\end{thm}

We use the same notation as in Section~3. In particular, $X=\bigsqcup(X_{t(n)},\delta^{t(n)}_G)$ and $Y=\bigsqcup(X_{t(n)},d^{t(n)})$ are coarse disjoint unions of copies of $M$ with warped and unwarped distance, respectively. Let $I:=\bigoplus_n B(L^2(M,\mu_{t(n)}))$ be an ideal in $\ControlledSupport[X]$ and the quotient map is denoted by $q:\ControlledSupport[X]\to \ControlledSupport[X]/I$. Since the action is free, we still have the isomorphism 
\begin{align}\label{isom crossed}
    \Psi:\ControlledSupport[Y]/I\rtimes_{\alg} G
    \longrightarrow
    \ControlledSupport[X]/I.
\end{align}

To prove the theorem, we construct a representation of the crossed
product in which the image of the coarse Laplacian has no spectral gap.
Our strategy is to tensor a covariant system containing a
\(G\)-invariant unit vector with a unitary representation of \(G\)
having almost invariant unit vectors but no nonzero invariant vectors.
 We denote by $\Delta:=\Delta_{E_r}$ the Laplacian defined in Lemma~\ref{Coarse Generating}. Let $\xi_n:=\frac{1}{\sqrt{\mu_{t(n)}(M)}}1_{Y_{t(n)}}$ be the normalized constant vector in $L^2(Y_{t(n)},\mu_{t(n)})$. We define a state $\omega_n$ on $\prod B(L^2(Y_n,\mu_{t(n)}))$ by
\begin{align*}
    \omega_n(T):=\langle \xi_n,T\xi_n \rangle=\langle \xi_n,T_n\xi_n \rangle
\end{align*}
for $T=(T_n)_n \in \prod B(L^2(Y_n,\mu_{t(n)}))$. Then we define a state $\omega$ by $\omega(T):=\varinjlim_{\mathcal{U}} \omega_n(T)$ for any non-principal ultrafilter $\mathcal{U}$. The state \(\omega\) descends to the quotient by \(\overline I\),
and we denote the induced state by $\omega$ again.
Let \((\mathcal H,\rho,\Omega)\) be the associated GNS representation.
We also denote by $\lambda$ the unitary representation $G\to U(L^2(Y))$ induced by the probability measure preserving action $\alpha$. Then for every $g\in G$, $[T\Omega]\mapsto [(\lambda(g)T)\Omega]$ defines a unitary representation  $U:G\to U(\HH)$. We restrict the representation $\rho$ to 
\begin{align*}
    \overline{\ControlledSupport[Y]}^{L^2}/\overline{I} \subseteq \prod B(L^2(Y_n,\mu_{t(n)}))/\overline{I},
\end{align*}
and still denote by $\rho$.
 Note that $(\rho, U,\HH)$ is a covariant system of the $C^*$-dynamical system $\ad:G\curvearrowright \left(\overline{\ControlledSupport[Y]}^{L^2}/\overline{I}\right)$. Now let $\sigma:G\to U(\HH')$ be an arbitrary unitary representation of $G$ and consider the $*$-homomorphism and unitary representation
\begin{align*}
    \rho\otimes I&:\left(\overline{\ControlledSupport[Y]}^{L^2}/\overline{I}\right)\to B(\HH\otimes \HH')\\
    U\otimes\sigma&:G\to U(\HH\otimes \HH').
\end{align*}
The triple $(\rho\otimes I, U\otimes \sigma, \HH\otimes\HH')$ is a covariant system and the associated representation on the algebraic crossed product $\left(\overline{\ControlledSupport[Y]}^{L^2}/\overline{I}\right)\rtimes_{\alg} G$ is denoted by $\widetilde{\rho}$. Since $\Delta$ has controlled support, we can write
\begin{align*}
    q(\Delta):=\sum_{g\in F} A_g g \in \ControlledSupport[X]/I \cong \left(\ControlledSupport[Y]/I\right)\rtimes_{\alg} G
\end{align*}
for a finite subset $F\subseteq G$ and $A_g \in \ControlledSupport[Y]/I$.

\begin{lem}\label{tensor}
In the above setting we have the following. 
\begin{enumerate}[i)]
    \item If there is a sequence of unit vectors $(\eta_n)_n$ in $\HH'$ with $\|\eta_n-\sigma(g)\eta_n\|\to 0$ for every $g\in G$, then we have that 
    \begin{align*}
        \|\widetilde{\rho}(q(\Delta))(\Omega\otimes \eta_n)\|\to 0.
    \end{align*}

    \item If $\HH'$ does not have $\sigma$-invariant vectors, then for any $\eta\in \HH'$, we have that 
    \begin{align*}
        \Omega\otimes \eta\perp \ker(\widetilde{\rho}(q(\Delta)))
    \end{align*}
\end{enumerate}
\end{lem}

\begin{proof}
\begin{enumerate}[i)]
    \item For each \(g\in F\), choose a lift
\(\widetilde A_g\in\mathbb C_{\mathrm{cs}}[Y]\) of \(A_g\).
Since each \(\xi_n\) is \(G\)-invariant, we have
    \begin{align*}
        \left\|\sum_{g\in F}\rho(A_g)\Omega\right\|^2=\varinjlim_{\mathcal{U}}\left\|\sum_{g\in F}\widetilde{A}_g\xi_n\right\|^2=\varinjlim_{\mathcal{U}}\left\|\sum_{g\in F}\widetilde{A}_g\lambda(g)\xi_n\right\|^2=\varinjlim_{\mathcal{U}}\|\Delta\xi_n\|^2=0.
    \end{align*}
    Therefore, 
    \begin{align*}
        \widetilde{\rho}(q(\Delta))(\Omega\otimes \eta_n)=\sum_{g\in F} \rho(A_g)\Omega \otimes \sigma(g)\eta_n=\sum_{g\in F} \rho(A_g)\Omega \otimes (\sigma(g)\eta_n-\eta_n) \to 0
    \end{align*}
    since $F\subseteq G$ is finite.
    \item
    As in \cite[Section~5]{Jeroen2021Geometric} (cf. \cite[Section~3]{Willett2014Geometric} for the discrete cases), for $T\in \ControlledSupport[X]$ supported on a controlled set $E$, we define a function $\Phi(T)$ by $\Phi(T)(x):=T1_{E_x}$.
    By \cite[Proposition~7.9]{Jeroen2021Geometric}, we have that
    \begin{align*}
        \ker (\widetilde{\rho}(q(\Delta)))=\left\{\xi\in \HH\otimes \HH':
        \begin{array}{l}
           \widetilde{\rho}(q(A-\Phi(A)))\xi=0 \\
             \forall A\in \ControlledSupport(X) \text{ such that }\Phi(A)\in L^{\infty}(X)
        \end{array}\right\}.
    \end{align*}
    We denote by $P_{\Omega}\in B(\HH)$ the orthogonal projection onto $\langle \Omega\rangle\subseteq \HH$. Take any $\xi \in \ker (\widetilde{\rho}(q(\Delta)))$ and write $(P_{\Omega}\otimes I)\xi=\Omega\otimes \xi'$ for some $\xi'\in \HH'$. Then since $\Phi(\lambda(g))=1$ for every $g\in G$, we have
    \begin{align*}
        \Omega\otimes \xi'=(P_{\Omega}\otimes 1)(U\otimes \sigma)(g) \xi=(U\otimes \sigma)(g)(P_{\Omega}\otimes 1)\xi=\Omega\otimes \sigma(g)\xi',
    \end{align*}
    where we used the fact that $U(g)P_{\Omega}=P_{\Omega}=P_{\Omega}U_g$ for all $g\in G$. This and the assumption that $\HH'$ has no invariant vectors imply that $\xi'=0$ and thus $(P_{\Omega}\otimes I)\xi=0$. Therefore,
    \begin{align*}
        \langle\xi,\Omega\otimes \eta \rangle=\langle\xi,(P_{\Omega}\otimes I)(\Omega\otimes \eta) \rangle=\langle(P_{\Omega}\otimes I)\xi,\Omega\otimes \eta \rangle=0.
    \end{align*}
\end{enumerate}
\end{proof}

Now we complete the proof of Theorem~\ref{if not kazhdan}.
\begin{proof}[Proof of Theorem~\ref{if not kazhdan}]
    Since we assume that $G$ does not have Kazhdan's property (T), there exists a unitary representation $\sigma$, which satisfies the assumptions of i) and ii) of Lemma~\ref{tensor}. Since $\Omega\otimes \eta_n$ are orthogonal to the kernel of $\widetilde{\rho}(q(\Delta))$ and satisfy $  \|\widetilde{\rho}(q(\Delta))(\Omega\otimes \eta_n)\|\to 0$, $\widetilde{\rho}(q(\Delta))$ does not have a spectral gap. By the isomorphism \eqref{isom crossed}, we can view $\widetilde{\rho}\circ q$ is a unital representation of $\ControlledSupport[X]$, in which
the image of $\Delta$ has no spectral gap. Therefore $X$ does not have geometric property (T).
\end{proof}

\section*{Acknowledgements}
The author thanks Jintao Deng for making comments on the paper.

\section*{AI Statement}
During the preparation of this work, the author used ChatGPT (OpenAI) for English-language editing and presentation, as well as for mathematical discussions concerning Section 4. The author reviewed and verified all mathematical content and takes full responsibility for the content of the article.
\bibliographystyle{alpha}
\bibliography{main}

\end{document}